\documentclass[11pt]{amsart}
\usepackage{amsfonts}
\usepackage{amssymb}
\usepackage{mathtools}
\mathtoolsset{showonlyrefs} 
\usepackage{nicefrac}
\usepackage{enumerate}
\usepackage{url}
\usepackage{color}      
\usepackage{epsfig}
\usepackage{graphicx}
\usepackage{tikz}
\usepackage{forest}
\usepackage{esint}
\usepackage{enumitem}  
\usepackage{comment}
\usepackage{appendix}   

\newtheorem{teo}{Theorem}[section]
\newtheorem{lem}[teo]{Lemma}

\theoremstyle{remark}
\newtheorem{re}[teo]{Remark}

\theoremstyle{definition}
\newtheorem{de}[teo]{Definition}

\newcommand{\CChi}{\mbox{\Large$\chi$}}

\renewcommand{\S}{\mathcal{S}}
\newcommand{\dd}{\text{\rm d}\mkern0.5mu}

\def\R{{\mathbb R}}

\def\N{{\mathbb N}}

\def\T{{\mathbb{T}_m}}

\title[Dirichlet-to-Neumann maps on Trees]{Dirichlet-to-Neumann maps on trees}

\author[L. M. Del Pezzo]{Leandro M. Del Pezzo}

\author[ N. Frevenza ]{Nicol\'{a}s Frevenza}
\author[ J. D. Rossi ]{Julio D. Rossi}
    \address{Leandro M. Del Pezzo, Nicol\'{a}s Frevenza and Julio D. Rossi
    \hfill\break\indent
    CONICET and Departamento  de Matem{\'a}tica, FCEyN, \hfill\break\indent
    Universidad de Buenos Aires,
    \hfill\break\indent Pabellon I, Ciudad Universitaria (1428),\hfill\break\indent
    Buenos Aires, Argentina.}
    \email{{\tt  ldelpezzo@utdt.edu, nfrevenza@dm.uba.ar, jrossi@dm.uba.ar}}

\begin{document}

\keywords{Dirichlet-to-Neumann map, Mean value formulas, Equations on trees.\\
\indent 2010 {\it Mathematics Subject Classification.} 35J05, 35R30, 31E05, 37E25.}

\begin{abstract} 
	In this paper we study the Dirichlet-to-Neumann 
	map for solutions to mean value formulas on trees. 
	We give two alternative definition of the Dirichlet-to-Neumann map.
	For the first definition (that involves the product of a 
	``{}gradient"{} with a ``{}normal vector"{}) and for a linear mean 
	value formula on the directed tree (taking into account only the 
	successors
	of a given node) we obtain that the Dirichlet-to-Neumann map is 
	given by $g\mapsto cg'$ (here $c$ is an explicit constant). 
	Notice that this is a local operator of order one. 
	We also consider linear undirected mean value formulas (taking into 
	account not only the successors but the ancestor and the successors 
	of a given node) and prove a similar result. For this kind
	of mean value formula we include some existence and uniqueness 
	results for the associated Dirichlet problem.
	Finally, we give an alternative definition 
	of the Dirichlet-to-Neumann map (taking into account differences 
	along a given branch of the tree). 
	With this alternative definition, for a certain range of 
	parameters, we obtain that the Dirichlet-to-Neumann map is given by 
	a nonlocal operator (as happens for the classical Laplacian in the 
	Euclidean space). 
\end{abstract}

\maketitle

\section{Introduction}\label{intro}

	Informally, the Dirichlet-to-Neumann map works as follows: given a 
	function $g$ on $\partial \Omega$, solve 
	the Dirichlet problem for the Laplacian with this datum inside the 	
	domain and then compute the normal derivative of the 
	solution on $\partial \Omega$ to obtain the operator $\Lambda (g)$. 
	Our main goal in this paper is to study the 
	Dirichlet-to-Neumann map for solutions to mean value formulas on 
	trees.

	The study of Dirichlet-to-Neumann maps for partial differential 
	equations (PDEs) has a rich history in the literature. For the 
	classical second order operator $\text{div} (a(x) Du) =0$ the 
	Dirichlet-to-Neumann map is related to the widely studied 
	Calderon's inverse problem, that is, knowing the Dirichlet-to-Neumann
	map, $g\mapsto \Lambda_a (g)$, find the coefficient $a(x)$ (see for instance \cite{C} and the 
	survey \cite{U}). 
	This problem has a well known application
	in electrical impedance tomography. 
	The Dirichlet-to-Neumann map is also related to 
	fractional powers of the Laplacian. For the classical Laplacian in 
	a half space it is well
	known that the Dirichlet-to-Neumann map gives the fractional 
	Laplacian (with power $1/2$), that is a nonlocal operator, 
	see \cite{CS}.

	Let us include now a brief comment on previous bibliography on 
	mean value formulas. Mean value formulas characterize solution 
	to certain PDEs. For example, in the Euclidean setting, 
	the validity of the mean value formula in balls characterize 
	harmonic functions.
	Nonlinear mean value properties that characterize solutions to 	
	nonlinear PDEs can also be found, for example, in \cite{MPR,rv,hr,hr1}. 
	These mean value properties reveal to be quite useful when 
	designing numerical schemes that approximate solutions to the 
	corresponding nonlinear PDEs, see \cite{O1,O2}. For mean values on 
	graphs (and trees) we refer to \cite{ary,KK,KW,KLW,Ober} and
	\cite{BBGS,s-tree,s-tree1} and references therein.

Linear and nonlinear mean value properties on trees are models that are close (and related to) 
to linear and nonlinear PDEs, 
hence it seems natural to look for the Dirichlet-to-Neumann map in the context of solutions to a mean value formula defined in a tree. The
analysis performed here can be viewed as just a first step into the study
of the general Calderon problem in a tree. 

It turns out that our first step in the analysis is to find a suitable definition for this Dirichlet-to-Neumann map on the tree. We have
two different definitions for this concept.
Our first definition starts with the idea of what is the normal derivative: we take the
``gradient'' of a function at a node and the inner product with a ``normal vector'', then
we multiply by a scaling parameter (a suitable power of the distance of the node to the root of the tree) and 
compute the limit as the node goes to the boundary of the tree (see the precise definitions below). 
This definition combined with the fact that we have an explicit formula for the solution of the 
Dirichlet problem for the case of the linear averaging operator, allow us to
explicitly compute the Dirichlet-to-Neumann map for smooth data. Surprisingly, this Dirichlet-to-Neumann map
just gives a local operator, $g\mapsto c g'$. 
When the mean value formula that we consider also depends on the ancestor
(that is, for an undirected tree), we have an alternative definition 
    of the Dirichlet-to-Neumann map. In this alternative definition we just consider the difference 
    between the values of $u$ at two successive vertices in a branch of the tree and then compute the limit
    as the vertices go to the boundary (suitable scaled). For this second definition the Dirichlet-to-Neumann map can 
be also computed (under a hypothesis on the parameter that measures the influence of the ancestor in the mean value formula) 
and gives rise to a more involved nonlocal operator that we also describe here.

\subsection{Notations and statements of the main results} 

	Let us first introduce some notations needed for the precise 
	statement of the results contained in this paper.

	Let $m\in\mathbb{N}_{\ge2}$. A tree $\T$ with regular 
    $m-$branching is a graph that consists of
   the empty set $\emptyset$ (also called the root of the tree) and all finite sequences $(a_1,a_2,\dots,a_k)$ with $k\in\N,$ 
    whose coordinates $a_i$ are chosen from $\{0,1,\dots,m-1\}.$ 
    
    \begin{center}
        \pgfkeys{/pgf/inner sep=0.19em}
        \begin{forest}
            [$\emptyset$,
                [0
                    [0
                        [0 [,edge=dotted]]
                        [1 [,edge=dotted]]
                        [2 [,edge=dotted]]
                    ]
                    [1
                        [0 [,edge=dotted]]
                        [1 [,edge=dotted]]
                        [2 [,edge=dotted]]
                    ]
                    [2
                        [0 [,edge=dotted]]
                        [1 [,edge=dotted]]
                        [2 [,edge=dotted]]
                    ]
                ]
                [1
                    [0
                        [0 [,edge=dotted]]
                        [1 [,edge=dotted]]
                        [2 [,edge=dotted]]
                    ]
                    [1
                        [0 [,edge=dotted]]
                        [1 [,edge=dotted]]
                        [2 [,edge=dotted]]
                    ]
                    [2
                        [0 [,edge=dotted]]
                        [1 [,edge=dotted]]
                        [2 [,edge=dotted]]
                    ]
                ]
                [2
                    [0
                        [0 [,edge=dotted]]
                        [1 [,edge=dotted]]
                        [2 [,edge=dotted]]
                    ]
                    [1
                        [0 [,edge=dotted]]
                        [1 [,edge=dotted]]
                        [2 [,edge=dotted]]
                    ]
                    [2
                        [0 [,edge=dotted]]
                        [1 [,edge=dotted]]
                        [2 [,edge=dotted]]
                    ]
                ]    
            ]
        \end{forest}
       
       A tree with $3-$branching.
    \end{center}
    
    The elements in $\T$ are called vertices or nodes. 
    Each vertex $x$ has $m$ successors, obtained by adding another coordinate.
    We will denote by 
    \[
        \S(x)=\{(x,i)\colon i\in\{0,1,\dots,m-1\}\}
    \]
    the set of successors of the vertex $x.$ If $x$ is not the root then $x$ has a only an immediate predecessor
    (or ancestor), which we will denote as $\hat{x}.$
    A vertex $x\in\T$ is called a $k-$level vertex 
    ($k\in\mathbb{N}$) if $x=(a_1,a_2,\dots,a_k)$.   
    The level of $x$ is denoted by $|x|.$ 
    The set of all $k-$level vertices is denoted by $\T\!\!\!^k.$
    Given $x\in \T$ such that $|x|>0$ and $j\in\{1,\dots,|x|-1\},$
    we denote by $x^{-j}$ the predecessor of $(|x|-j)-$level of $x.$ 
    In this context, we use the following notation $x^0=x.$
    
    A branch $\pi$ of $\T$ is an infinite sequence of vertices, each followed by an immediate successor.
    The collection of all branches forms the boundary of $\T$,  denoted by $\partial\T$.
    We can observe that the mapping $\psi:\partial\T\to[0,1]$ defined as
    \[
        \psi(\pi)\coloneqq\sum_{k=1}^{+\infty} \frac{a_k}{m^{k}}
    \]
    is surjective, where $\pi=(a_1,\dots, a_k,\dots)\in\partial\T$ and
    $a_k\in\{0,1,\dots,m-1\}$ for all $k\in\mathbb{N}.$ Whenever
    $x=(a_1,\dots,a_k)$ is a vertex, we set
    \[
        \psi(x)\coloneqq\psi(a_1,\dots,a_k,0,\dots,0,\dots).
    \]
    We can also associate to a vertex $x$ an
    interval $I_x$ of length $\tfrac{1}{m^{|x|}}$ as follows
    \[
        I_x\coloneqq\left[\psi(x),\psi(x)+\frac1{m^{|x|}}\right].
    \]
    Observe that for all $x\in \T$, $I_x \cap \partial\T$ is the
    subset of $\partial\T$ consisting of all branches that pass through $x$.
    In addition, for any branch $\pi=(a_1,\dots, a_k,\dots)\in\partial\T,$ we can
    associate to $\pi$ the sequence of intervals $\{I_{\pi,k}\}$ given by
    \[
        I_{\pi,k}\coloneqq I_{x_k}\quad \text{with }x_k=(a_1,\dots,a_k)
    \]
    for any $k\in\mathbb{N}.$ For any $k\in\mathbb{N},$ it is easy to check that  
    $I_{\pi,k+1}\subset I_{\pi,k}$ and $\psi(\pi)\in I_{x_k}.$

    Given a function $u:\T \to \mathbb{R}$ on the tree we define its gradient as
    the vector that encodes all the differences between $u(x)$ with the values of $u$ at the successors, $u(x,j)$, 
    that is, we let
    \[
        \nabla u (x) \coloneqq (u(x,0) - u(x),..., u(x,m-1) - u(x)),\qquad\forall x\in\T.
    \]
    
    Now, let us introduce the mean value formulas that we are interested in.
    Given $0\le \beta\le1,$ we say that a function $u\colon\T\to\R$ is a $\beta-$harmonic function on $\T$ if
    \[
        u(\emptyset)=\frac1m\sum_{j=0}^{m-1}u(j)
    \]
    and 
    \[
        u(x)=\beta u(\hat{x})+\dfrac{1-\beta}{m}\sum_{i=0}^{m-1} u(x,i),\qquad\forall x\in\T\setminus\{\emptyset\}.
    \]
    We shall often write harmonic function as an abbreviation for $0-$harmonic function. Notice that for $\beta =0$ we have
    that harmonic functions are solutions to
    \begin{equation} \label{lapla.usual}
    0= \sum_{i=0}^{m-1} \left(u(x,i) - u(x)\right) = \langle \nabla u (x), (1,...,1) \rangle .
    \end{equation}
    
    Note that if $\beta=\frac{1}{m},$  the definition of a $\beta-$harmonic function coincides with the classic mean value
    on the tree viewed as a graph (the value of $u$ at any node $x$ is just the mean value of $u$ at all the nodes that are connected to $x$). 
    Whereas for $\beta=0,$ 
    the definition coincides with the definition of a harmonic function
    on the arborescence (also called directed) tree. In this last case the equation involve only
    the values of $u$ at a node and its successors. See, for instance, 
    \cite{anandam,DPMR1,DPMR2,DPMR3,KLW,KW}.

    Given a bounded function  $g:[0,1]\to\R,$ 
    we say that a function $u$ is a solution to the 
    $\beta-$Dirichlet problem with boundary datum $g$ if $u$ is a $\beta-$harmonic function and verifies
    \[
        \lim_{k\to+\infty} u(x_k)= g(\psi(\pi)), 
        \qquad \forall \pi=(x_1, \dots, x_k, \dots)\in\partial\T.
    \]
    
    Our first result shows that for a continuous datum $g$
    there is existence and uniqueness of solutions  for the Dirichlet problem
    when $\beta<\tfrac12$, while when $\beta \geq \tfrac12$ 
    there are no non-constant bounded solutions. 
    Hence the Dirichlet problem with a continuous datum $g$ is solvable only when $\beta <\tfrac12$ 
    or when $g$ is a constant function.

    \begin{teo} \label{teo.harmonicas2} 
        Let $g:[0,1]\to\R$ be a continuous function.
        For any $0\le\beta<\tfrac12$,   
        there is a unique bounded solution to 
        the Dirichlet problem with boundary datum $g,$
        that is there is a unique bounded $\beta-$harmonic function $u_g$ such that
        \begin{equation} \label{BDP}
                 \lim\limits_{k\to+\infty} u_g(x_k)= 
                g(\psi(\pi)), \qquad \forall \pi=(x_1, \dots, x_k, \dots)\in\partial\T.
        \end{equation}
        Moreover, this solution can be explicitly computed and is given by
        \begin{equation} \label{formula.harmonica2}
            u_g(x)=
                \begin{cases}
                    \displaystyle\fint_{I_{x }} g(t) \dd t & \text{if }
                    \beta=0,\\[10pt]
                    \displaystyle p^{|x|} \fint_0^1 g(t) \dd t +  
                    \sum_{j=0}^{|x|-1} p^j(1-p) \fint_{I_{x^{-j}}} g(t) \dd t , & \text{if }
                    \beta\in(0,\tfrac12),
                \end{cases}
        \end{equation} 
        where $p = \tfrac{\beta}{1-\beta}\in(0,1)$ and $\fint$ denote the average integral, $\fint_A g = \tfrac{1}{|A|} \int_A g$.

        Finally, for any $\beta\geq 1/2$ every bounded $\beta-$harmonic function is constant.
    \end{teo}	
    
    Formula \eqref{formula.harmonica2} allows us to obtain existence and uniqueness 
    of solutions for more general boundary data, see Remark \ref{rem.g.masgemeral}.

    As an interesting property of the solutions we notice that for this 
    notion of $\beta-$harmonic 
    functions, in the case $\beta\in(0,\tfrac12)$ 
    we have a strong comparison principle 
    (this property does not hold in the case $\beta=0,$ that is, 
    for solutions to the usual Laplacian on the arborescence tree).

    \begin{teo} \label{teo.harmonicas2.strongCP} 
        For any $0<\beta<\tfrac12$,  
        two bounded $\beta-$harmonic function $u,v$  
        that are ordered,
        \[
            u(x) \leq v(x) \qquad \forall x \in \T ,
        \]
        and touch at one vertex, $u(x_0)=v(x_0)$ for some $x_0\in \T$, must coincide in the whole tree, that is, \[u\equiv v \quad \text{ on } \T.\]
    \end{teo}
    
    \medskip

    After proving the existence and uniqueness of a solution with the explicit formula 
    \eqref{formula.harmonica2}, we are ready to introduce our first version of the Dirichlet-to-Neumann map.
    
    \begin{de}
         Let $0\le\beta<\tfrac1{2},$ and fix a vector
         $\eta = (\eta_0,\dots,\eta_{m-1})\in\R^m$ 
         (the ``{}normal vector"{}). 
         The Dirichlet-to-Neumann map for $\beta-$harmonic functions 
         in the direction of $\eta$, that we call $\Lambda_{\beta,\eta}\colon C^2([0,1])\to C^1([0,1])$ 
         is defined by
         \[
            \Lambda_{\beta,\eta}(g)(\psi(\pi))\coloneqq
                \lim\limits_{k\to\infty} m^{k}\langle \nabla u_g(x_k),\eta \rangle
         \]
         for any $\pi=(x_1,\dots,x_k,\dots)\in\partial\T.$ Here $u_g$ is the $\beta-$harmonic function on $\T$ with boundary value $g$.
    \end{de}
    
    In the case $\beta=0$, 
    we obtain the following explicit expression 
    for the Dirichlet-to-Neumann map depending on the usual derivative $g'.$
    
    \begin{teo}\label{teo:intro.carDN1.1}
        Let $\beta=0,$ $g\in C^2([0,1]),$ $\eta = (\eta_0,\dots,\eta_{m-1})\in\R^m$ 
        and $u_g$ be the solution of the Dirichlet problem ($u_g$ is a harmonic function in $\T$) with boundary datum $g.$ 
        Then for any $\pi=(x_1, \dots, x_k,\dots)\in \partial\T,$ we have
        \[
               \Lambda_{0,\eta}(g)(\psi(\pi))=g'(\psi(\pi))\langle \eta,\omega_m\rangle,
        \] 
        where 
        \[
                \omega_m = 
                \begin{cases}
                    \left(\tfrac{1-m}{2m},\tfrac{3-m}{2m},\dots, 
                    \tfrac{-2}{2m}, 0, 
                    \tfrac{2}{2m},\dots,\tfrac{m-3}{2m},\tfrac{m-1}{2m}\right)
                    &\text{if } m \text{ is odd},\\[5pt]
                    \left(\tfrac{1-m}{2m},\tfrac{3-m}{2m},\dots, \tfrac{-2}{2m}, \tfrac{2}{2m},\dots,\tfrac{m-3}{2m},\tfrac{m-1}{2m}\right)
                     &\text{if } m \text{ is even}.
                \end{cases}
        \]
    \end{teo}
    
    We remark that 
    $
    \langle \omega_m; (1,...,1) \rangle =0.
    $
    This orthogonality is natural since from \eqref{lapla.usual} we have $\langle \nabla u (x) ; (1,...,1) \rangle =0$. 
    
    In the case of $\beta\neq0$, we need to add an extra assumption to obtain 
    the explicit expression for the Dirichlet-to-Neumann map.
    \begin{teo}\label{teo:carDN1}
        Let $0<\beta<\tfrac1{2},$ $g\in C^2([0,1]),$ 
        $\eta = (\eta_0,\dots,\eta_{m-1})\in\R^m$  be such that
        \[
         \langle \eta, (1,...,1)\rangle =   \sum_{j=0}^{m-1}\eta_j=0,
        \]
        and $u_g$ be the solution of the Dirichlet problem ($u_g$ is $\beta-$harmonic in $\T$) with 
        boundary datum $g.$ 
        Then, for any $\pi=(x_1, \dots, x_k,\dots)\in \partial\T,$ we have 
        \[
            \Lambda_{\beta,\eta}(g)(\psi(\pi))=\dfrac{(1-2\beta)}{m(1-\beta)}
               g'(\psi(\pi))\langle \eta,\varpi_m\rangle ,
        \]
        where $\varpi_m=(0,1,\dots,m-1).$ 
    \end{teo}
    
    \begin{re}
        If $\beta=0$ and $\displaystyle\sum_{j=0}^{m-1}\eta_j=0$ then
        \begin{align*}
            \Lambda_{0,\eta}(g)(\psi(\pi))&=g'(\psi(\pi))\langle \eta,\omega_m\rangle =g'(\psi(\pi))\sum_{i=0}^{m-1}\eta_i\dfrac{2i+1-m}{2m}\\
               &=g'(\psi(\pi))\sum_{i=0}^{m-1}\eta_i\left(i + \dfrac{1-m}{2m}\right) =g'(\psi(\pi))\sum_{i=0}^{m-1}\eta_i i\\
               &=g'(\psi(\pi))  \langle \eta,(0,1,\dots,m-1)\rangle.
        \end{align*}
    \end{re}
    
    Notice that with our first definition the Dirichlet-to-Neumann map $ \Lambda_{\beta,\eta}(g)$ is a local operator
    of order one ($ \Lambda_{\beta,\eta}(g)$ is just a constant times $g'$). However, in the Euclidean setting the Dirichlet-to-Neumann
    map is a nonlocal operator (also of order one).
    
    Now we present an alternative definition 
    of the Dirichlet-to-Neumann map. In this alternative definition we just consider the difference 
    between the values of $u$ at two successive vertices in a branch of the tree and then compute the limit
    as the vertices go to the boundary (suitable scaled).

    \begin{de}
         The Dirichlet-to-Neumann map associated to $\beta-$harmonic functions, that we call
         $\Gamma_{\beta}\colon C^2([0,1])\to L^\infty(0,1)$ 
         is defined by
         \[
            \Gamma_{\beta}(g)(\psi(\pi))\coloneqq
                \lim\limits_{k\to\infty} p^{-|x_k|}\left(u_g(x_{k+1})-u_g(x_k)\right),
         \]
         for any $\pi=(x_1,\dots,x_k,\dots)\in\partial\T.$ Here the scaling parameter $p$ is given by $p=\tfrac{\beta}{1-\beta}.$ 
    \end{de}
    
    With this definition, when $\tfrac{1}{m+1}<\beta<\tfrac1{2}$, the Dirichlet-to-Neumann map turns out to be a nonlocal operator
    (as in the Euclidean setting).
    
    \begin{teo}\label{teo:carDN2}
        Let $\tfrac{1}{m+1}<\beta<\tfrac1{2},$ $g\in C^2([0,1]),$ 
            and $u_g$ be the solution of the Dirichlet problem for $\beta-$harmonic functions with 
            boundary datum $g.$ 
            Then, for any $\pi=(x_1, \dots, x_k,\dots)\in \partial\T,$ we get 
            \[
               \Gamma_{\beta}(g)(\psi(\pi))=(1-p)\int_0^1\mathcal{K}(\pi,t)(g(\psi(\pi))-g(t))\dd t
            \]
            where the kernel $\mathcal{K}\colon \partial\T\times[0,1]\to \mathbb{R}$ is given by
            \[
                \mathcal{K}(\pi,t)\coloneqq         
                    1+m\dfrac{1-p}{m-p}\left(\left(\dfrac{m}{p}\right)^{\mathcal{N}(\psi(\pi),t)}-1\right)
                    \CChi_{I_{\pi,1}}(t)
            \]
            with $\mathcal{N}(\psi(\pi),\cdot)\colon I_{\pi,1}\to\mathbb{N}$ defined by
            \[
                \mathcal{N}(\psi(\pi),\cdot)\coloneqq\max\{k\in\mathbb{N}\colon t\in I_{\pi,k}\}. 
            \]
    \end{teo}
    
    The reason why we require the condition $\beta > \tfrac{1}{m + 1}$ is that we need that
$p = \tfrac{\beta}{1-\beta} > \tfrac{1}{m}$ in our arguments.
    
	In \cite{DPMR1,DPMR2,DPMR3} it was considered the Dirichlet problem on the directed tree for a general averaging operator. 
	The results for the Dirichlet problem presented in this paper, where the harmonic 
	equation at point $x$ depends on the predecesor of $x$ (except for $x=\emptyset$), can be easily adapted 
	for nonlinear averaging operators similar to those studied in \cite{DPMR1,DPMR2,DPMR3}. Unfortunately, in the 
	nonlinear case we can not find a general explicit formula for the Dirichlet-to-Neumann map 
	due to the fact that in the nonlinear case we do not have an explicit expression for the solution 
	of the Dirichlet problem like the one find for the lineal case.

    \noindent{\bf Organization of the paper.}  First, 
    in Section \ref{section:cp}, we prove a comparison principle;
    then, in Section \ref{DPS} we prove Theorems  
    \ref{teo.harmonicas2} and \ref{teo.harmonicas2.strongCP}; 
    and, finally, in Section \ref{section:DNlinealcase} we prove 
    Theorems \ref{teo:intro.carDN1.1}, \ref{teo:carDN1} and \ref{teo:carDN2}.
    
\section{A comparison principle}\label{section:cp}
    Let us start by introducing the definition of 
    $\beta-$subharmonic and  $\beta-$superharmonic functions. 
    
    Given $\beta\in[0,1],$  
    a function $u\colon\T\to\R$ is called a $\beta-$subharmonic function on $\T$ if the following 
    inequalities hold
    \begin{align*}
         u(\emptyset)&\le \frac1m\sum_{j=0}^{m-1}u(j),\\
         u(x)&\le\beta u(\hat{x})+\dfrac{1-\beta}{m}
        \sum_{i=0}^{m-1} u(x,i)\quad\forall x\in\T\setminus\{\emptyset\},
    \end{align*}
    and $u$ is a $\beta-$superharmonic 
    function if the opposite inequalities hold. Thus, if $u$ is  both $\beta-$subharmonic and 
    $\beta-$superharmonic, then $u$ is $\beta-$harmonic.
    
    Before showing our comparison principle we need to prove the following lemma (that gives the validity of
    a maximum/comparison principle).
    
    \begin{lem}\label{lem:cp}
        Let $u$ and $v$ be $\beta-$sub and $\beta-$superharmonic functions
        respectively, and\\
       $f,g\colon[0,1]\to[-\infty,\infty]$ be given by
        \[
            f(\psi(\pi))\coloneqq
            \sup \left\{
            \limsup_{x_k\to \pi}u(x)\colon \{x_k\}_{k\in\mathbb{N}}
            \subset\T,
             x_k\to \pi\right\}, 
        \]
        \[
            g(\psi(\pi))\coloneqq \inf \left\{
            \liminf_{x_k\to \pi}v(x)\colon \{x_k\}_{k\in\mathbb{N}}
            \subset\T, x_k\to \pi\right\}, 
        \]
       for any $\pi\in\partial \T.$ 
       If $g$ is finite at every point,then
        \[
            \sup\{u(x)-v(x)\colon x\in\T\}
            \le \sup\{f(t)-g(t)\colon t\in[0,1]\}.
        \]
    \end{lem}
    
    \begin{proof}
        Let $$M= \sup\{u(x)-v(x)\colon x\in\T\}.$$ 
        
        If $M=\infty,$ then there is a sequence 
        $\{x_k\}_{k\in\mathbb{N}}\subset\T$ such that
        $x_k\to\pi\in \partial\T$ and
        \[
            \infty=\lim_{k\to\infty} u(x_k)-v(x_k)
            \le \limsup_{k\to\infty}u(x_k) -
            \liminf_{k\to\infty} v(x_k)\le f(\psi(\pi))-g(\psi(\pi))
        \]
        Therefore, 
        \[
        	\sup\{f(t)-g(t)\colon t\in[0,1]\}=\infty,
        \] 
        and there is nothing to prove.

        Throughout of the rest of this proof, we assume that
        $M<\infty.$
        For any $\varepsilon>0$ there is $x_0\in\T$ such that
        \[
            M-\varepsilon\le u(x_0)-v(x_0).
        \]
        If $x_0\neq\emptyset,$ without loss of generality, we can assume that
        $u(\hat{x}_0)-v(\hat{x}_0)<M-\varepsilon.$
        
        Observe that $u-v$ is a $\beta-$subharmonic function since 
        $u$ and $v$ are bounded $\beta-$sub and superharmonic functions, respectively.
        Then, the sequence $\{x_k\}_{k\in \N}$ defined by
        \[
            x_k\in \S(x_{k-1})\text{ such that } u(x_k)-v(x_k)=\max\{u(z)-v(z)\colon z\in\S(x_{k-1})\},
        \]
        satisfies 
        \[
            M-\varepsilon \le u(x_k)-v(x_k), \qquad \forall k\in\N.
        \]
        Additionally, there is $\pi\in\partial\T$ such that $x_k\to\pi.$ Then
        \[
            M-\varepsilon \le \limsup_{k\to\infty}u(x_k)-v(x_k)\le f(\psi(\pi))-g(\psi(\pi)).
        \]
        Therefore, for any $\varepsilon>0$
        \[
            M-\varepsilon\le \sup\{f(t)-g(t)\colon t\in[0,1]\}.
        \]
        Since $\varepsilon$ is arbitrary, the proof is complete.
    \end{proof}
    
    As an immediate corollary of the previous lemma, we have the  following the comparison principle.
    
    \begin{teo}\label{teo:cp}
         Let $u$ and $v$ be $\beta-$sub and $\beta-$superharmonic functions, respectively, and
         define
       $f,g\colon[0,1]\to[-\infty,\infty]$ by
        \[
            f(\psi(\pi))\coloneqq
            \sup \left\{
            \limsup_{x_k\to \pi}u(x)\colon \{x_k\}_{k\in\mathbb{N}}
            \subset\T,
             x_k\to \pi\right\} ,
        \]
        \[
            g(\psi(\pi))\coloneqq \inf \left\{
            \liminf_{x_k\to \pi}v(x)\colon \{x_k\}_{k\in\mathbb{N}}
            \subset\T, x_k\to \pi\right\}, 
        \]
       for any $\pi\in\partial \T.$ Assume that $f$ and $g$ are finite at each point.
        If $$f(t)\le g(t)$$ for all $t\in[0,1]$ then $$u(x)\le v(x)$$ for all $x\in\T.$
    \end{teo}

\section{The Dirichlet Problem. Existence and uniqueness}\label{DPS}
    This section is devoted to the proofs of Theorems \ref{teo.harmonicas2} and \ref{teo.harmonicas2.strongCP}.
    
    \subsection{Existence and uniqueness}
    In the case $\beta=0$ the proof of Theorem \ref{teo.harmonicas2} can be found in any of the references
    \cite{DPMR1}, \cite{DPMR2} or \cite{s-tree}, for this reason throughout this section we assume that $\beta\neq0.$
    
    We first study the case in which $g$ is a characteristic function.
    \begin{lem}\label{lem:caracteristicas}
        Let $0<\beta<\tfrac12,$ $n\in\mathbb{N},$ $j\in\{0,1,\dots,m^n-1\}$
        and $I_{n,j}=[\tfrac{j}{m^n},\tfrac{j+1}{m^n}].$ There exists a
        $\beta-$harmonic function such that for any
        $\pi=(x_1, \dots, x_k, \dots)\in\partial\T$ hold
        \begin{equation}
            \label{eq:caracteristicas}
              \lim_{k\to\infty} u_{n,j}(x_k)
              =
              \begin{cases}
              \CChi_{I_{n,j}}(\psi(\pi))
                 & \text{if }
                \psi(\pi)\not\in\{\tfrac{j}{m^n},\tfrac{j+1}{m^n}\},\\
                1&\text{if }\psi(\pi)\in
                \{\tfrac{j}{m^n},\tfrac{j+1}{m^n}\}
                \text{ and }\exists k_0 /
                \psi(x_k)\in I_{n,j}\forall k\ge k_0,\\
                0&\text{if }\psi(\pi)\in
                \{\tfrac{j}{m^n},\tfrac{j+1}{m^n}\}
                \text{ and }\psi(x_k)\not \in I_{n,j}\forall k.
              \end{cases}
        \end{equation}
    \end{lem}
    
    \begin{proof}
         We assume that $j=0$, the other cases can be handled in an analogous way. 
         We set $I_n=I_{n,0}.$
        
        If $n=0$ then $u\equiv 1$ is a  $\beta-$harmonic function that satisfies \eqref{eq:caracteristicas}.
        
        \medskip
        
        If $n=1,$ we define
        \[
            w_1(x)=
            \begin{cases}
                1 &\text{if } x=\emptyset,\\
                \left(\dfrac{\beta}{1-\beta}\right)^{|x|} &\text{if } x^{-(|x|-1)}\neq (\emptyset,0),\\               
                b_{|x|}&\text{if } x^{-(|x|-1)}= (\emptyset,0),
            \end{cases}
        \]
        where $\{b_i\}_{i\in\mathbb{N}}$ is the sequence given by
        \begin{align*}
            & b_1=m-(m-1)\dfrac{\beta}{1-\beta},\\
            & b_2=\dfrac{b_1-\beta}{1-\beta},\\
            & b_{i+1}=\dfrac{b_{i}-\beta b_{i-1}}{1-\beta}=b_i+\dfrac{\beta}{1-\beta}(b_i-b_{i-1}), \qquad\forall i>2.
        \end{align*}
        
        We can check by a direct calculation that $w$ is a $\beta-$harmonic function. Moreover,
        for any $\pi=(x_1, \dots, x_k, \dots)\in\partial\T$ such that $\psi(\pi)\not\in I_{1}$
        \[
             \lim_{k\to\infty} w_1(x_k)=   \lim_{k\to\infty} \left(\frac{\beta}{1-\beta}\right)^{k}=0=\CChi_{I_1}(\psi(\pi))
        \]
        due to the hypothesis $0<\beta<\tfrac{1}{2}.$ 
        In addition, there is 
        $\pi_0=(y_1,\dots,y_k,\dots)\in\partial\T$ 
        such that $\psi(\pi_0)=\tfrac1m,$ $y_1=1$ and 
         \[
             \lim_{k\to\infty} w_1(y_k)=   \lim_{k\to\infty} \left(\frac{\beta}{1-\beta}\right)^{k}=0.
        \]
        
        On the other hand, it is easy to check that $\{b_i\}_{i\in\mathbb{N}}$ is an
        increasing and bounded Cauchy sequence. Then, there exists $b>0$ such that $b_i\to b.$ Thus,
        for any $\pi=(x_1, \dots, x_k, \dots)\in\partial\T$ such that $\psi(\pi)\in I_{1}$ and $x_1=0$, we have
        \[
             \lim_{k\to\infty} w_1(x_k)=   \lim_{k\to\infty} b_{|x_k|}= \lim_{k\to\infty} b_{k}= b.
        \]
        
        Hence, taking $u_1\coloneqq\tfrac{w_1}{b}$, we have a bounded $\beta-$harmonic function such that \eqref{eq:caracteristicas}
        holds.
               
        Now, if $n>1$ and $z\in \T\!\!\!^n$ is such that $I_n=I_{z},$ we let
        \[
            w_n(x)\coloneqq
            \begin{cases}
                w_{n-1}(x) & \text{ if } |x|<n \text{ or } I_{x}\cap I_{\hat{z}}=\emptyset,\\
                w_{n-1}(\hat{z})\left(\dfrac{\beta}{1-\beta}\right)^{|x|-n+1} &
                \text{ if } |x|\ge n, I_{x}\cap I_{\hat{z}}\not=\emptyset \text{ and } 
                x^{-(|x|-n)}\not=z,\\
                b_{n,|x|-n+1}&
                \text{ if } |x|\ge n, x^{-(|x|-n)}=z,\\
            \end{cases}
        \]
        where $\{b_{n,i}\}_{j\in\mathbb{N}}$ is the sequence given by
        \begin{align*}
            & b_{n,1}=
            w_{n-1}(\hat{z})
            \left(\dfrac{m-(m-1)\beta}{1-\beta}\right)-m\dfrac{\beta}{1-\beta}w_{n-1}(z^{-2}),\\
            & b_{n,2}=\dfrac{b_{n,1}-\beta w(\hat{z})}{1-\beta},\\
            & b_{n,i+1}=\dfrac{b_{n,i}-\beta b_{n, i-1}}{1-\beta}=b_{n,i}+\dfrac{\beta}{1-\beta}(b_{n,i}-b_{n,i-1}), \qquad\forall i>2.
        \end{align*}
 
        By an inductive argument we obtain that the following statements hold: 
        \begin{itemize}
            \item The function $w_n$ is a $\beta-$harmonic function;
            \item For any $\pi=(x_1, \dots, x_k, \dots)\in\partial\T$ 
                such that $\psi(\pi)\not\in I_{n}$
                \[
                    \lim_{k\to\infty} w_n(x_k)= 0=\CChi_{I_n}(\psi(\pi))
                \]
                due to $0<\beta<\tfrac{1}{2};$ 
            \item There is $\pi_0=(y_1,\dots,y_k,\dots)\in\partial\T$ 
                such that $\psi(\pi_0)=\tfrac1{m^n},$ $y_n=(a_1,\dots,a_{n-1},1)$ with 
                $a_i=0$, $1\le i\le n-1$ and
                \[
                    \lim_{k\to\infty} w_1(y_k)=   \lim_{k\to\infty} \left(\frac{\beta}{1-\beta}\right)^{k}=0;
                \]
            \item  $\{b_{n,i}\}_{i\in\mathbb{N}}$ is an increasing and bounded Cauchy sequence with $b_{n,1}>0$. 
        \end{itemize}
        Then, there is $b>0$ such that $b_{n,i}\to b.$ 
        Thus, for any $\pi=(x_1, \dots, x_k, \dots)\in\partial\T$ such that $\psi(\pi)\in I_{m}$
        \[
             \lim_{k\to\infty} w_1(x_k)=   \lim_{k\to\infty} b_{n,k-n+1}=b.
        \]
        
        Finally, the function $u_n\coloneqq\tfrac{w_n}{b}$ is a bounded $\beta-$harmonic function such that \eqref{eq:caracteristicas}
        holds.
    \end{proof}
        
        Now, we can show existence and uniqueness for a general continuos boundary datum.
    
    \begin{teo}\label{teo:existuni} 
       For any $0<\beta<\tfrac12,$ 
       the Dirichlet problem for $\beta-$harmonic functions with continuous boundary datum 
       $g\colon [0,1]\to \R$, has a unique solution.
    \end{teo}
    \begin{proof} 
        Let $g\colon [0,1]\to \R$ be
        a continuous boundary datum.
        
        Let us start by observing that, for any constant $c$, a function 
        $u$ is a solution of the 
        Dirichlet problem for $\beta-$harmonic 
        functions with datum $g$,
        if only if $u+c$ is a solution of the 
        Dirichlet problem for $\beta-$harmonic 
        functions with datum $g+c$. 
        Therefore, without loss of generality, we may 
        assume that $g$ is a nonnegative function. 
        
        Let $\mathcal{A}$ be the set 
        \begin{equation}\label{teo:existuni:setA}
            \mathcal{A} \coloneqq 
            \left\{
                \begin{aligned}
                    v \colon\T\to\mathbb{R} \colon 
                    &v\text{ is a bounded } \beta-\text{subharmonic function},\\
                    &\limsup_{k\to \infty}v(x_k)\le g(\psi(\pi)) \quad \forall \pi=(x_1,\dots,x_k,\dots)
                    \in\partial\T
                \end{aligned}
           \right\},
        \end{equation} 
        Note that $\mathcal{A}$ is not the empty set 
        since $v_0(x) \coloneqq \min\{g(t)\colon t\in[0,1]\}\in\mathcal{A}.$
        
        We claim that the function 
        \[
            u^*\colon\T\to\mathbb{R}, \quad u^*(x) \coloneqq \sup\{ v(x)\colon v\in \mathcal{A}\} 
        \] 
        is a $\beta-$subharmonic function. Indeed, for any  $v\in\mathcal{A}$ we have 
        \[
            v(\emptyset)\le\dfrac1m\sum_{j=0}^{m-1}v(\emptyset,j) \le\dfrac1m\sum_{j=0}^{m-1}u^*(\emptyset,j),
        \]
        \[
            v(x) \leq \beta v(\hat{x}) +  \frac{1-\beta}{m}\sum_{j=0}^{m-1} v(x,j) 
            \leq \beta u^*(\hat{x}) +  \frac{1-\beta}{m}\sum_{j=0}^{m-1} u^*(x,j),
            \qquad \forall x\in\T\setminus\{\emptyset\}.
        \] 
        Therefore, using the definition of $u^*$ as the supremum of the functions $v$, we get
        \[
            u^*(\emptyset)\le\dfrac1m\sum_{j=0}^{m-1}u^*(j),
        \]
        \[
            u^*(x)\le  \beta u^*(\hat{x}) +  \frac{1-\beta}{m}\sum_{j=0}^{m-1} u^*(x,j)\qquad \forall x\in\T\setminus\{\emptyset\},
        \]
        that is, $u^*$ is a $\beta-$subharmonic function. Now, we extend $u^*$ to the boundary by its $\limsup$.
        
        Our second claim is that for any 
        $\pi=(x_1,\dots,x_k,\dots)\in\partial\T$
        \[
            \lim_{k\to\infty}u^*(x_k)= g(\psi(\pi)).
        \]
        
        Let $\pi_0=(y_1,\dots,y_k,\dots)\in\partial\T.$ 
        For any $n\in\T,$
        there exist $j\in\{0,\dots,m-1\}$ and 
        $k_0$ such that
        $\psi(y_k)\in I_{n,j}=[\tfrac{j}{m^n},\tfrac{j+1}{m^n}]$
        for all $k\ge k_0.$
        Now taking $c=\min\{g(t)\colon t\in I_{n,j}\}$
        and $w_{n,j}=cu_{n,j}$ where $u_{n,j}$ is given by
        Lemma \ref{lem:caracteristicas}, we have that
        $w_{n,j}$ is a $\beta-$harmonic function such that
        \[
            \lim_{k\to\infty}w_{n,j}(x_k)\le g(\psi(\pi)),
            \qquad \forall\pi=(x_1,\dots,x_k,\dots)
                    \in\partial\T.
        \]
        Here, we are using that $g\ge 0$ in [0,1]. 
        Then, $w_{n,j}\in\mathcal{A},$ and 
        therefore $w_{n,j}(x)\le u^*(x)$ for any $x\in\T.$
        In particular, $w_{n,j}(y_k)\le u^*(y_k)$
        for any $k.$ Therefore,
        \[
           \min\{g(t)\colon t\in I_{n,j}\} =\lim w_{n,j}(y_k)\le \liminf_{k\to\infty}u^*(y_k).
        \]
        Taking the limit as $n\to \infty,$ 
        we have
        \[
           g(\psi(\pi_0))\le \liminf_{k\to\infty}u^*(y_k)
        \]
        since $g$ is a continuous function.
        
        On the other hand, taking
        $w^{n,j}=a(1-u_{n,j})+bu_{n,j}$ where 
        $a=2\max\{g(t)\colon t\in[0,1]\}$ and
        $b=\max\{g(t)\colon t\in I_{n,j}$ we obtain a $\beta-$harmonic function with 
        \[
            g(\psi(\pi))\le  \lim_{k\to\infty}w^{n,j}(x_k) ,
            \qquad \forall\pi=(x_1,\dots,x_k,\dots)
                    \in\partial\T.
        \]
        Thus, by Theorem \ref{teo:cp}, we have that 
        $v(x)\le w^{n,j}(x)$ for any $x\in\T.$ Therefore
        $u^*(x)\le w^{n,j}(x)$ for any $x\in\T.$
        In particular, $u^*(y_k)\le w^{n,j}(y_k) $
        for any $k.$ Then
        \[
            \limsup_{k\to\infty}u^*(y_k)\le
            \lim w_{n,j}(y_k)
            =\max\{g(t)\colon t\in I_{n,j}\}.
        \]
        Again, taking the limit as $n\to \infty,$ 
        we have
        \[
            \limsup_{k\to\infty}u^*(y_k)\le g(\psi(\pi_0)).
        \]
        Therefore
        \[
            \lim_{k\to\infty}u^*(y_k)= g(\psi(\pi_0)).
        \]
        
        Our last claim is that $u^*$ is a $\beta-$harmonic function. To verify this claim, we suppose, 
        arguing by contradiction, that $u^*$ is not a $\beta-$harmonic function. Then,
        there is $x_0\in \T$ such that 
        \[
            u^*(x_0)<
            \begin{cases}
                \dfrac1m\displaystyle\sum_{j=0}^{m-1} u^*(\emptyset,j) & 
                \text{if } x_0=\emptyset,\\
                \beta u^*(\hat{x}_0) + \displaystyle \frac{1-\beta}{m}\sum_{j=0}^{m-1} u^*(x_0,j) & 
                \text{if } x_0\neq\emptyset.
            \end{cases}
        \]
        Therefore, there exists $\delta>0$ such that for this $x_0\in \T$,
        \[
            u^*(x_0)+\delta\le  
            \begin{cases}
                \dfrac1m\displaystyle\sum_{j=0}^{m-1} u^*(\emptyset,j) & 
                \text{if } x_0=\emptyset,\\
                \beta u^*(\hat{x}_0) +  \displaystyle\frac{1-\beta}{m}\sum_{j=0}^{m-1} u^*(x_0,j) & 
                \text{if } x_0\neq\emptyset.
            \end{cases}
        \]
        Then, the function $v\colon\T\to\mathbb{R}$ defined by
        \[
            v(x)\coloneqq
                \begin{cases}
                    u^*(x) &\text{if }x\in\T\setminus\{x_0\},\\
                    u^*(x_0)+\delta &\text{if }x=x_0,\\
                \end{cases}
        \]
        belongs to $\mathcal{A}.$ Therefore, by the definition of $u^*,$ we have $v(x)\le u^*(x)$
        for any $x\in\T.$ In particular
        \[
            u^*(x_0)<u^*(x_0)+\delta=v(x_0)\le u^*(x_0).
        \]
        This contradiction establishes our last claim.
        
        Therefore,  from our last two claims, we have that $u^*$ is a solution to the $\beta-$Dirichlet problem with continuous 
        boundary datum $g.$
        
        Finally, we want to observe that the uniqueness of solution is, once more, a consequence of Theorem \ref{teo:cp}.
    \end{proof}

    We now give a different proof of the existence of solution finding an explicit formula for the solution. This explicit formula will be very
    useful when computing the Dirichlet-to-Neumann map.
    
    \begin{teo} \label{teo:formula-sol-DP-nolocal} 
    	Let $0<\beta<\tfrac12.$ 
        If $g$ is a continuous function the solution to the 
        $\beta-$Dirichlet problem with boundary datum $g$ is 
        given by the formula: 
        \begin{equation} \label{sol-no-local} 
            u(x)= p^{|x|} \fint_0^1 g(t) \dd t 
            + \sum_{j=0}^{|x|-1} p^j(1-p) 
            \fint_{I_{x^{-j}} } g(t) \dd t,
        \end{equation} 
        where $p = \tfrac{\beta}{1-\beta}\in(0,1)$.
    \end{teo}
    
    \begin{proof}
        We first observe that
        \begin{align*}
            u(\emptyset)&=\fint_0^1 g(t) \dd t
            =\frac{1}m\sum_{j=0}^{m-1}p\fint_0^1 g(t) \dd t+(1-p)\fint_0^1 g(t) \dd t\\
            &=\frac{1}m\sum_{j=0}^{m-1}\left[p\fint_0^1 g(t) \dd t+(1-p)\fint_{I_{j}} g(t) \dd t \right]
            =\frac{1}m\sum_{j=0}^{m-1}u(\emptyset,j).
        \end{align*}
        
        On the other hand, for any $x\neq \emptyset$ 
        \[
            u(x) = pu(\hat{x}) + (1-p)\fint_{I_x} g(t) \dd t.
        \] 
        Then
        \begin{align*} 
            \beta u(\hat{x}) + \frac{1-\beta}{m} \sum\limits_{j=0}^{m-1} u(x,j) 
            & =  \frac{\beta}{p} 
            \left[
                u(x) - (1-p){\fint_{I_{x}}} g(t) \dd t
            \right] \\
            & \qquad + \displaystyle\dfrac{1-\beta}{m} \sum\limits_{j=0}^{m-1} 
            \left[
                p u(x)+ (1-p)\fint_{I_{(x,j)}} g(t) \dd t
            \right]\\
            & =  u(x) - (1-p)(1-\beta) {\fint_{I_{x}}}g(t) \dd t \\
           & \qquad + (1-p)(1-\beta) {\fint_{I_{x}}}g(t) \dd t\\
            & =  u(x).
        \end{align*} 
        Therefore, $u$ is a $\beta-$harmonic function.

        To end the proof, we need to show that 
        \[
            \lim\limits_{k\to+\infty} u(x_k)= 
            g(\psi(\pi)) \quad \forall \pi=(x_1, \dots, x_k, \dots)\in\partial\T.
        \] 
        Fix $\varepsilon>0$ and $\pi\in\partial\T$. By continuity there exists $\delta>0$ such that 
        \[
            |g(t)-g(\psi(\pi))|<\varepsilon\quad\text{if }|t-\psi(\pi)|<\delta.
        \]
        Now, we choose $k_0$ such that 
        \[ 
            \sum_{j>k_0+1} p^j(1-p)<\varepsilon 
        \]
        which is possible since $0<p<1$. 
        Then, we pick $x\in \T$ such that 
        \[
            I_x \subset I_{x^{-1}} \subset \dots \subset I_{x^{-k_0}} 
            \subset (\psi(\pi)-\delta,\psi(\pi)+\delta) \quad  
            \text{ and } \quad p^{|x|}<\varepsilon.
        \] 
        Using both facts and the formula \eqref{sol-no-local}, it is easy to compute that
        \begin{align*}
            |u(x) &- g(z)| 
            = \left|
                \sum_{j=0}^{|x|-1} p^j(1-p) {\fint_{I_{x^{-j}}} } 
                    \left[
                        g(t)-g(\psi(\pi))
                    \right] \dd t  
                        + 
                        p^{|x|} \fint_0^1 
                            \left[
                                g(t)-g(\psi(\pi))
                            \right] \dd t 
                \right| \\
            &\leq 
                    \sum_{j=0}^{k_0} p^j(1-p) 
                    {\fint_{I_{x^{-j}}} } |g(t)-g(\psi(\pi))| \dd t 
                    + \displaystyle 
                    \sum_{j=k_0+1}^{|x|-1} 2p^j(1-p) \|g\|_{\infty}  
                    +  2p^{|x|}\|g\|_{\infty} \\
            &\leq \varepsilon(1+4\|g\|_{\infty}).
        \end{align*} 
        Therefore, as $\varepsilon$ is arbitrary it follows that 
        \[
            \lim_{x\to \pi } 
        u(x)=g(\psi(\pi)).
            \qedhere
        \]
    \end{proof}
    
    \begin{re} \label{rem.martingala}
    {\rm The function $u$ is indeed a martingale relative to the natural filtration determined in $\T_m$ 
    as the proof of the explicit formula \eqref{sol-no-local} shows. }
    \end{re}
    
    \begin{re} \label{rem.g.masgemeral}
        {\rm The formula \eqref{sol-no-local} makes sense for a bounded and integrable function $g$. 
        In fact, the function defined in \eqref{sol-no-local} is $\beta-$harmonic, so, the proof of Theorem \ref{teo:formula-sol-DP-nolocal} holds 
        with the clarification that in the boundary $u$ coincides with $g$ in any Lebesgue point of $g$.
        The proof of the comparison principle holds also in this case, so, the Dirichlet problem 
        with a bounded and integrable boundary condition has also a unique solution. }
    \end{re}
    
    In what follows we will analyze the case $\beta\in[\tfrac12,1]$. 
    For $\beta=1$ it is immediate that the only $1-$harmonic functions are the constants
    (indeed, by an inductive argument in the distance from $x$ to the root of the tree one can show that $u(x) =u(\emptyset)$ for every $x\in\T$). 
    For $\beta\in[\tfrac12,1)$ we will show that any $\beta-$harmonic 
    function on $\T,$ that is not constant is unbounded. 
    In both cases, the $\beta-$Dirichlet problem with a general (non-constant) continuous boundary condition has not solution.

    \begin{teo} \label{teo:nonexist}
        Let $\beta\in[\tfrac12,1)$ and $u$ a non-constant $\beta-$harmonic function. Then, $u$ is unbounded.
    \end{teo}
    \begin{proof}
        We assume without loss of generality that $u(\emptyset)=1$. 
        Let $y$ be a point of $\T$ such that $u(y)\neq 1$ and $|y|=\min\{|x|\colon u(x)\neq 1\}.$  
        Such point exists since $u$ is non-constant and if there is more than one, choose any of them. 
        Denote $x_0$ for its father and observe that $u(x_0)=1$. 
        By the mean-value formula there is a successor of $x_0$, called $x_1$, such that $u(x_1)=1 + a_1$ 
        with $a_1>0$. Now, using again the mean-value formula for $x_1$ we have  
        \begin{equation}
            \frac{1}{m} \sum_{j=0}^{m-1} u(x_1,j) = 1 + \frac{a_1}{1-\beta}>1 +a_1.
        \end{equation} 
        So, there is a descendent of $x_1$, called $x_2$, such that $u(x_2)=1 + a_2$ with $a_2>a_1$. 
        Suppose that there are $x_1, \dots, x_n\in \T$, such that $x_i$ is the father of 
        $x_{i+1}$ and $u(x_i) = 1+a_i$ with $a_i>a_{i-1}$, 
        we want to construct the next point in the sequence. 
        By the mean-value formula 
        \begin{equation}\label{eq-beta-mayor-a-1/2}
            \frac{1}{m} \sum_{j=0}^{m-1} u(x_n,j) = 1 + a_n + \frac{\beta}{1-\beta}(a_n - a_{n-1})> 1 +a_n, 
        \end{equation} 
        so, there exists $x_{n+1}$ such that $\hat{x}_{n+1}=x_n$ and $u(x_{n+1})=1+a_{n+1}$ with $a_{n+1}>a_n$. 
        Moreover, by \eqref{eq-beta-mayor-a-1/2} it is easy to check that 
        \begin{equation*}
            a_{n+1} - a_{n} \geq \frac{\beta}{1-\beta} (a_{n} - a_{n-1})\geq 
            \left(\frac{\beta}{1-\beta}\right)^{n-1} (a_{2} - a_{1}).
        \end{equation*} 
        Thus, for $\beta \in [\tfrac12,1)$ we have that $u(x_n) = 1 + a_n \to \infty$ as $n\to\infty$.
    \end{proof}
    
    Observe that, having proved Theorems \ref{teo:existuni} , \ref{teo:formula-sol-DP-nolocal}, and \ref{teo:nonexist},
    we obtained Theorem \ref{teo.harmonicas2}.

    \subsection{Strong comparison principle}
    To end this section, we prove a strong comparison principle in the case $\beta\in(0,\tfrac12)$ and
    show that this does not hold for $\beta=0$.

    \begin{teo} 
        Let  $\beta \in(0,\tfrac12),$ $u,v$ $\beta-$harmonic 
        functions such that $u\leq v$ on $\T$ and $v(x_0) = u(x_0)$
        for some $x_0\in \T$. 
        Then, $$v\equiv u$$ in the whole tree.
    \end{teo}
    
    \begin{proof} 
        Since $u,v$ are $\beta-$harmonic functions,  
        \[
            0 = u(x_0) - v(x_0) =
            \frac{1}{m} \sum_{j=0}^{m-1} [ u(x_0,j) - v(x_0,j) ] \le0,\quad \text{ if } x_0=\emptyset,
        \] 
        \[
            0 = u(x_0) - v(x_0) = \beta (u(\hat{x}_0) - v(\hat{x}_0)) + \frac{(1-\beta)}{m} 
            \sum_{j=0}^{m-1} [ u(x_0,j) - v(x_0,j) ],\quad \text{ if } x_0\neq\emptyset.
        \] 
        Then, each term in the sums must be zero due to $u\leq v$ on $\T,$ so \[u(x_0,j) - v(x_0,j)= u(\hat{x}_0) - v(\hat{x}_0) = 0,\] 
        for all $j\in\{0,\dots, m-1\}$. Hence 
        \begin{equation*} u\equiv v. \qedhere
        \end{equation*}
    \end{proof}

    For $\beta=0$ the strong comparison principle fails. Indeed, consider for instance $m=3$ and 
    $u\colon\mathbb{T}_3 \to\mathbb{R}$ the unique solution of the Dirichlet problem with 
    boundary condition $g$, where $g$ is a nonnegative continuous function with $\text{supp}(g)=[\tfrac{2}{3},1]$ and
    $$\int_{\tfrac23}^1 g(t) \dd t =1.$$
    then, using the formula for the solution $$u(x)= \fint_{I_x} g(t) \dd t,$$ 
    it follows that $u(\emptyset,0)=0$ but $u\not\equiv 0$ because $u(\emptyset,2)=1$. Notice that this example shows that 
    the strong maximum principle also fails (there are nontrivial nonnegative harmonic functions that vanish in an interior point).

\section{Dirichlet-to-Neumann maps}\label{section:DNlinealcase}

In this section we deal with the computation of the formulas for the two versions of the Dirichlet-to-Neumann maps
described in the introduction.

\subsection{First definition}
First, we deal with our first definition and consider
\[
            \Lambda_{\beta,\eta}(g)(\psi(\pi))\coloneqq
                \lim\limits_{k\to\infty} m^{|x_k|}\langle \nabla u_g(x_k),\eta \rangle
         \]
         for any $\pi=(x_1,\dots,x_k,\dots)\in\partial\T.$ Here $u_g$ is the $\beta-$harmonic function on $\T$ with boundary value $g$.

    \subsubsection{Case $\beta=0$}
        Let $\beta=0,$ $g\in C^2([0,1]),$ $\eta = (\eta_0,\dots,\eta_{m-1})\in\R^m$ 
        and $u$ be the solution of the Dirichlet problem with 
        boundary datum $g.$ 
        Given $\pi=(x_1, \dots, x_k,\dots)\in \partial\T,$ we now want to compute 
        \[
             \Lambda_{0,\eta}(g)(\psi(\pi))
             =\lim_{k\to\infty} m^{|x_k|}\langle \nabla u(x_k),\eta \rangle.
        \]
        
        \begin{lem}\label{lem:DNA1}
            Let $\beta=0,$ $g\in C^2([0,1]),$ $\eta = (\eta_0,\dots,\eta_{m-1})\in\R^m$ 
            and $u$ be the solution of the Dirichlet problem with 
            boundary datum $g.$ 
            Then for any $x\in\T$ we get
            \[
            \langle \nabla u(x),\eta \rangle= \dfrac{g'(t_x)}{m^{|x|}}\langle \eta,\omega_m\rangle+ O(m^{-2|x|}),
            \]
            where $t_x$ is the midpoint of $I_x$ and 
            \begin{equation}\label{omega-m}
                 \omega_m = 
                \begin{cases}
                    \left(\tfrac{1-m}{2m},\tfrac{3-m}{2m},\dots, \tfrac{-2}{2m}, 0, 
                    \tfrac{2}{2m},\dots,\tfrac{m-3}{2m},\tfrac{m-1}{2m}\right)
                    &\text{if } m \text{ is odd},\\[5pt]
                     \left(\tfrac{1-m}{2m},\tfrac{3-m}{2m},\dots, \tfrac{-2}{2m}, \tfrac{2}{2m},\dots,\tfrac{m-3}{2m},\tfrac{m-1}{2m}\right)
                     &\text{if } m \text{ is even}.
                \end{cases}
            \end{equation}
        \end{lem}
        
        \begin{proof}
            First, since $u$ is a harmonic function, given $x\in\T$, we get
            \begin{align*}
                \langle \nabla u(x),\eta \rangle =& 
                \sum_{i=0}^{m-1} \eta_i \left( u(x,i) - \frac{1}{m} \sum_{j=0}^{m-1} u(x,j)\right) \\
                = & \sum_{i=0}^{m-1} \eta_i u(x,i) - \frac{1}{m} 
                \sum_{j=0}^{m-1} \left(\sum_{i=0}^{m-1} \eta_i\right) u(x,j)\\
                = & \dfrac1m\sum_{j=0}^{m-1}  \left(m \eta_j - \sum_{i=0}^{m-1} \eta_i\right)  u(x,j).
            \end{align*} 
            To abbreviate we introduce the notation 
            \[
                c_j(\eta)\coloneqq m \eta_j - \sum_{i=0}^{m-1} \eta_i, \qquad j\in\{0,\dots,m-1\}
            \]
            and we have
             \begin{align*}
                \langle \nabla u(x),\eta \rangle = \dfrac1m\sum_{j=0}^{m-1}  c_j(\eta)  u(x,j).
            \end{align*} 
            On the other hand, by Theorem \ref{teo.harmonicas2}, we have that
            \[
                u(x,j)=\fint_{I_(x,j)} g(t) \dd t, \qquad \forall j\in\{0,\dots,m-1\}.
            \]
            Then, 
            \begin{align*}
                \langle \nabla u(x),\eta \rangle 
                = & \dfrac1m\sum_{j=0}^{m-1} c_j(\eta)  u(x,j)\\
                = & \dfrac1m\sum_{j=0}^{m-1}c_j(\eta)\fint_{I_{(x,j)}} g(t) \dd t\\
                = & m^{|x|}\sum_{j=0}^{m-1}c_j(\eta)\int_{I_{(x,j)}} g(t) \dd t.
            \end{align*} 
            
            Next, let us set $t_x$ as the midpoint of $I_x,$ that is
            \[
                t_x=\psi(x)+\dfrac{1}{2m^{|x|}}.
            \]
            
            Since $g\in C^2([0,1]),$ we now that for any $t\in (\psi(x),\psi(x)+\tfrac1{m^{|x|}})$ 
            there is a $\xi_{t_x,t}\in I_x$ such that
            \[
                g(t)=g(t_x)+g'(t_x)(t-t_x)+g''(\xi_{t_x,t})\dfrac{(|t-t_x|^2)}{2}.
            \]
            Then,
            \begin{equation}
                \label{eq:DN1}
    	        \begin{aligned}
                    &\langle \nabla u(x),\eta \rangle = m^{|x|} 
                    \sum_{j=0}^{m-1} c_j(\eta) \int_{I_{(x,j)}} g(t) d t \\ 
                    & \qquad =  m^{|x|} \sum_{j=0}^{m-1} c_j(\eta) 
                    \left( \frac{g(t_x)}{m^{|x|+1}} + g'(t_x) \int_{I_{(x,j)}} (t-t_x) \dd t + 
                    \frac{1}{2} \int_{I_{(x,j)}} g''(\xi_{t_x,t})(t-t_x)^2 \dd t \right).
                \end{aligned}  
            \end{equation}
            
            Observe that 
           \begin{equation}
                \label{eq:DN2}
                \begin{aligned}
                     m^{|x|} \sum_{j=0}^{m-1} c_j(\eta)  \frac{g(t_x)}{m^{|x|+1}}&= \frac{g(t_x)}{m} 
                    \sum_{j=0}^{m-1} c_j(\eta)  \\
                    &=\frac{g(t_x)}{m} \sum_{j=0}^{m-1} \sum\limits_{i=0}^{m-1} (\eta_j - \eta_i)\\
                    &= 0,
                \end{aligned}
            \end{equation}
            and
            \[
                \frac{m^{|x|}}{2}\left|\int_{I_{(x,i)}} g''(\xi_{t_x,t})(t-t_x)^2 \dd t \right|\le
                \dfrac{\|g''\|_{\infty}}{24}\dfrac1{m^{2|x|}}.
            \]
           Thus
            \begin{equation}
                \label{eq:DN3}
                 \frac{m^{|x|}}{2}\int_{I_{(x,i)}} g''(\xi_{t_x,t})(t-t_x)^2 \dd t=O( m^{-2|x|}).
            \end{equation}
                
            On the other hand, note that if $m$ is odd then for $j=0, \dots, \tfrac{m-3}2$ we have that 
            \begin{equation} \label{igualdad-integrales} 
                \int_{I_{(x,j)}} (t-t_x) \dd t = - \int_{I_{(x,m-1-j)}} (t-t_x)  \dd t 
            \end{equation} 
            and 
            \[
                \int_{I_{\left(x,\tfrac{m-1}{2}\right)}} 
                (t-t_x) \dd t = 0.
            \]
            
            Now, assuming that $m$ is odd, by \eqref{eq:DN1}, \eqref{eq:DN2}, \eqref{eq:DN3} and
            \eqref{igualdad-integrales} we obtain 
            \begin{align*}
                \langle \nabla u(x),\eta \rangle &= m^{|x|} g'(t_x) \sum_{j=0}^{\tfrac{m-3}2} 
                (c_j(\eta) - c_{m-1-j}(\eta)) \int_{I_{(x,j)}} (t-t_x) \dd t  + O(m^{-2|x|}) \\
                &=  m^{|x|} g'(t_x) \sum_{j=0}^{\tfrac{m-3}2} 
                \dfrac{\eta_j - \eta_{m-1-j}}{2m^{2|x|} } 
                \left( \frac{2j+1}{m} -1 \right) + O(m^{-2|x|})\\
                &= \frac{g'(t_x)}{2m^{|x|}}  \sum_{j=0}^{\tfrac{m-3}2} 
                (\eta_j - \eta_{m-1-j}) \left( \frac{2j+1}{m} -1 \right)+  O(m^{-2|x|})\\
                &=\dfrac{g'(t_x)}{m^{|x|}}\langle \eta,\omega_m\rangle+  O(m^{-2|x|}),
            \end{align*}
            where $\omega_m = (\tfrac{1-m}{2m},\tfrac{3-m}{2m},\dots, \tfrac{-2}{2m}, 0, \tfrac{2}{2m},\dots,\tfrac{m-3}{2m},\tfrac{m-1}{2m}).$

            On the other hand, if  $m$ is even then \eqref{igualdad-integrales} 
            holds for $j=0,\dots, \tfrac{m-2}2$. Then, proceeding as in the case that $m$ is 
            odd, we get
            \[
                \langle \nabla u(x),\eta \rangle= 
                \dfrac{g'(t_x)}{m^{|x|}}\langle \eta,\omega_m\rangle+  
                O(m^{-2|x|}),
            \]
             where $\omega_m = 
             \left(\tfrac{1-m}{2m},\tfrac{3-m}{2m},\dots, 
             \tfrac{-1}{2m}, 
             \tfrac{1}{2m},\dots,\tfrac{m-3}{2m},\tfrac{m-1}{2m}
             \right)$.
        \end{proof}
        
        Now, we can prove Theorem \ref{teo:intro.carDN1.1}.
        
        \begin{proof}[Proof of Theorem \ref{teo:intro.carDN1.1}]
            Let $\beta=0,$ $\eta = (\eta_1,\dots,\eta_m)\in\R^m$ 
            and $u$ be the solution of the Dirichlet problem with 
            boundary datum $g\in C^2([0,1]).$ 
            
            Given $\pi=(x_1,\dots,x_k,\dots)\in\partial\T,$ by Lemma \ref{lem:DNA1}, we get
             \[
                \langle \nabla u(x_k),\eta \rangle= \dfrac{g'(t_k)}{m^{|x_k|}}\langle \eta,\omega_m\rangle
                +  O(m^{-2|x_k|}),
            \]
            where $t_k$ is the midpoint of $I_{x_k}.$ 
            Then, since $g\in C^2([0,1])$ and 
            $\pi=(x_1,\dots,x_k,\dots)$ we get $t_k\to\psi(\pi)$ as $k\to\infty$ and
            \begin{align*}
                \Lambda_{0,\eta}(g)(\psi(\pi))
                &=\lim_{k\to \infty} m^{|x_k|} \langle \nabla u(x_k),\eta \rangle\\
                &= \lim_{k\to\infty} g'(t_k)\langle \eta,\omega_m\rangle 
                +  O(m^{-|x_k|})\\
                &=g'(\psi(\pi))\langle \eta,\omega_m\rangle.
                \qedhere
            \end{align*}
        \end{proof}
    
    \subsubsection{Case $0<\beta<\tfrac1{2}$}
        Now, we prove a explicit formula for the Dirichlet-to-Neumann map assuming that $\displaystyle\sum_{j=0}^{m-1}\eta_j=0$.

        \begin{proof}[Proof of Theorem \ref{teo:carDN1}]
            For any $x_k\in\T\setminus\{\emptyset\}$ and $i\in\{0,\dots,m-1\},$
            by \eqref{sol-no-local}, we get
            \[
			    u(x_k,i)=p u(x_k) + (1-p)
			    \fint_{I_{(x_k,i)}} g(t) \dd t.
		    \]
		    Then, since $\displaystyle \sum_{i=0}^{m-1}\eta_i=0,$ we have
		    \begin{align*}
		        \sum_{i=0}^{m-1}\eta_i(u(x_k,i)-u(x_k))&=
			    (p-1)u(x_k)\sum_{i=0}^{m-1}\eta_i
			    +(1-p)\sum_{i=0}^{m-1}\eta_i\fint_{I_{(x_k,i)}} g(t) \dd t\\
			    &=(1-p)\sum_{i=0}^{m-1}\eta_i\fint_{I_{(x_k,i)}} g(t) \dd t.
		    \end{align*}
		    
		    Since $g\in C^2([0,1]),$ we now that for any $t\in I_x$ 
            there is a $\xi_{\psi(\pi),t}\in I_x$ such that
            \[
                g(t)=g(\psi(\pi))+g'(\psi(\pi))(t-\psi(\pi))+
                g''(\xi_{\psi(\pi),t})\dfrac{(|t-\psi(\pi)|^2)}{2}.
            \]
		    Therefore, using again that 
		    $\displaystyle\sum_{i=0}^{m-1}\eta_i=0,$ 
		    we obtain
		    \begin{align*}
			    \sum_{i=0}^{m-1}\eta_i(u(x_k,i)-u(x_k))&=
			    (1-p)\sum_{i=0}^{m-1}\eta_i\fint_{I_{(x_k,i)}} 
			    (g(t)-g(\psi(\pi))) \dd t\\
			    &=(1-p)g'(\psi(\pi))
			    \sum_{i=0}^{m-1}\eta_i\fint_{I_{(x_k,i)}} 
			    (t-\psi(\pi)) \dd t+O(m^{-2|x_k|})\\
			    &=(1-p)g'(\psi(\pi))
			    \sum_{i=0}^{m-1}\eta_i
			    \left(\psi(x_k,i)+\dfrac1{2m^{|x_k|+1}}\right)
			    +O(m^{-2|x_k|})\\
			    &=(1-p)g'(\psi(\pi))
			    \sum_{i=0}^{m-1}\eta_i
			    \left(\psi(x_k)+\dfrac{i}{m^{|x_k|+1}}\right)
			    +O(m^{-2|x_k|})\\
			    &=(1-p)g'(\psi(\pi))
			    \sum_{i=0}^{m-1}\eta_i
			    \dfrac{i}{m^{|x_k|+1}}
			    +O(m^{-2|x_k|}).
		    \end{align*}
		    Thus
		    \[
		        \lim_{k\to\infty} m^{|x_k|}\sum_{i=0}^{m-1}\eta_i(u(x_k,i)-u(x_k))
		        =\frac{1-p}{m}g'(\psi(\pi))
			    \sum_{i=0}^{m-1} i \eta_i .
		        \qedhere
            \]
        \end{proof}

    \subsection{Second definition}
    Now we consider our second definition for the Dirichlet-to-Neumann map
     \[
            \Gamma_{\beta}(g)(\psi(\pi))\coloneqq
                \lim\limits_{k\to\infty} p^{-|x_k|}\left(u_g(x_{k+1})-u_g(x_k)\right),
         \]
         for any $\pi=(x_1,\dots,x_k,\dots)\in\partial\T.$
    
    \subsubsection{Case $\tfrac{1}{m+1}<\beta<\tfrac12$} 
        Given $x\in\T\setminus\{\emptyset\}$ and 
        $t\in I_{x^{-(|x|-1)}},$ we define $n(x,t)$ as follows
        \begin{itemize}
            \item If $t\in I_x$ then $n(x,t)\coloneqq |x|;$
            \item If $t\not\in I_x$ then $n(x,t)$ is the only number in $\{1,\dots,|x|-1\}$ such that $t\in I_{x^{-(|x|-n(x,t))}}$ and $t\not \in I_{x^{-(|x|-n(x,t) +1)}}.$
        \end{itemize}
        In fact
        \[
            n(x,t)= 
            \begin{cases}
                |x| &\text{if } t\in I_x;\\
                |x|-1 & \text{if } t\in I_{x^{-1}}\setminus I_x;\\
                \quad\vdots&\\
                \quad 1& \text{if } t\in I_{x^{-(|x|-1)}}\setminus I_{x^{-(|x|-2)}}.\\
            \end{cases}
        \]
        
        We are now ready to prove a technical lemma that will be relevant throughout the rest of this section.
        \begin{lem}\label{lemma:tcl}
            Let $g\in C^2([0,1]),$ 
            and $u$ be the solution of the Dirichlet problem with $0<\beta<\tfrac12$ and boundary datum $g.$
            Then for any $x\in\T\setminus\{\emptyset\},$ and $j\in\{0,\dots,m-1\}$
            \[
                u(x,j)-u(x)=-(1-p)\int_0^1 \mathcal{K}^j_m(x,t)g(t) \dd t   
            \]
            where
            \[
                \mathcal{K}^j_m(x,t)=
                \begin{cases}
                    p^{|x|}-\dfrac{m(1-p)}{m-p}
                        \left[p^{|x|}-m^{|x|}\right]-m^{|x|+1} &\text{if } t\in I_{(x,j)},\\[15pt]
                      p^{|x|}\left(1-m\dfrac{(1-p)}{m-p}
                        \left[1-\left(\dfrac{p}{m}\right)^{-n(x,t)}
                        \right]\CChi_{I_{x^{-(|x|-1)}}}(t)\right) 
                        &\text{if } t\not \in I_{(x,j)}.
                \end{cases}
            \]
        \end{lem}
        \begin{proof}
            Given $x\in\T$ and $i\in\{0,\dots,m-1\},$ 
            by \eqref{formula.harmonica2}, we get
            \begin{align*}
                u(x)&=p^{|x|}\fint_0^1 g(t) \dd t +  
                        \sum_{i=0}^{|x|-1} p^i(1-p) \fint_{I_{x^{-i}}} g(t) \dd t\\
                u(x,j) &=p^{|x|+1}\fint_0^1 g(t) \dd t +  
                        \sum_{i=0}^{|x|} p^i(1-p) \fint_{I_{(x,j)^{-i}}} g(t) \dd t\\
                        &= p^{|x|+1}\fint_0^1 g(t) \dd t +  
                        \sum_{i=1}^{|x|} p^i(1-p) \fint_{I_{x^{-i+1}}} g(t) \dd t
                        +(1-p)\fint_{I_{(x,j)}} g(t)\dd t\\
                         &= p^{|x|+1}\fint_0^1 g(t) \dd t +  
                        \sum_{i=0}^{|x|-1} p^{i+1}
                        (1-p) \fint_{I_{x^{-i}}} g(t) \dd t
                        +(1-p)\fint_{I_{(x,j)}} g(t)\dd t.
            \end{align*}
            Then
            \begin{align*}
                &u(x,j)-u(x)=\\
                &=-p^{|x|}(1-p)\fint_0^1 g(t) \dd t -  
                        \sum_{i=0}^{|x|-1} p^{i}
                        (1-p)^2 \fint_{I_{x^{-i}}} g(t) \dd t
                        +(1-p)\fint_{I_{(x,j)}} g(t)\dd t\\
                &=-(1-p)
                \int_{0}^1 
                    \left\{
                        p^{|x|}+\sum_{i=0}^{|x|-1}
                        p^{i}(1-p)m^{|x|-i}\CChi_{I_{x^{-i}}}(t)
                        -m^{|x|+1}\CChi_{I_{(x,j)}}(t)
                    \right\}g(t) \dd t\\
                 &=-(1-p)m^{|x|}
                \int_{0}^1 
                    \left\{
                        \left(\dfrac{p}{m}\right)^{|x|}+(1-p)\sum_{i=0}^{|x|-1}
                        \left(\dfrac{p}{m}\right)^{i}\CChi_{I_{x^{-i}}}(t)
                        -m\CChi_{I_{(x,j)}}(t)
                    \right\}g(t) \dd t.
            \end{align*}

            Observe that, if $t\in I_{(x,j)}$
            \begin{align*}
                \left(\dfrac{p}{m}\right)^{|x|}+(1-p)\sum_{i=0}^{|x|-1}
                        \left(\dfrac{p}{m}\right)^{i}\CChi_{I_{x^{-i}}}(t)
                        &-m\CChi_{I_{(x,j)}}(t)=\\
                        &=\left(\dfrac{p}{m}\right)^{|x|}-\dfrac{m(1-p)}{m-p}
                        \left[\left(\dfrac{p}{m}\right)^{|x|}-1\right]-m.
            \end{align*}
            
            Whereas if $t\in I_{x^{-(|x|-1)}}\setminus I_{(x,j)}$ 
            then 
            \begin{align*}
                \left(\dfrac{p}{m}\right)^{|x|}+(1-p)\sum_{i=0}^{|x|-1}
                        \left(\dfrac{p}{m}\right)^{i}&\CChi_{I_{x^{-i}}}(t)
                        -m\CChi_{I_{(x,j)}}(t)=\\&=
                        \left(\dfrac{p}{m}\right)^{|x|}+(1-p)\sum_{i=|x|-n(x,t)}^{|x|-1}
                        \left(\dfrac{p}{m}\right)^{i}\CChi_{I_{x^{-i}}}(t)\\
                        &=\left(\dfrac{p}{m}\right)^{|x|}-m\dfrac{(1-p)}{m-p}
                        \left[\left(\dfrac{p}{m}\right)^{|x|}-\left(\dfrac{p}{m}\right)^{|x|-n(x,t)}\right].
            \end{align*}
            
            Finally, if $t\in[0,1]\setminus  I_{x^{-(|x|-1)}}$ then
             \[
                \left(\dfrac{p}{m}\right)^{|x|}+(1-p)\sum_{i=0}^{|x|-1}
                        \left(\dfrac{p}{m}\right)^{i}\CChi_{I_{x^{-i}}}(t)
                        -m\CChi_{I_{(x,j)}}(t)=\left(\dfrac{p}{m}\right)^{|x|}.
            \]

            Therefore,
            \[
                u(x,j)-u(x)=-(1-p)\int_0^1 \mathcal{K}^j_m(x,t)g(t) \dd t.
            \qedhere
            \]
        \end{proof}
        
        \begin{re}\label{remark:ik0}
            Observe that in the proof of the previous lemma we showed that
            \[
                m^{-|x|}\mathcal{K}^j_m(x,t)
                =\left(\dfrac{p}{m}\right)^{|x|}+(1-p)\sum_{i=0}^{|x|-1}
                        \left(\dfrac{p}{m}\right)^{i}\CChi_{I_{x^{-i}}}(t)
                        -m\CChi_{I_{(x,j)}}(t) \quad \forall x\in \T, t\in[0,1]. 
            \]
            Then, it holds that
            \begin{align*}
                  \int_0^1 \mathcal{K}_m(x,t) \dd t ={m^{|x|}}\int_{0}^1 
                    \Big\{
                        \left(\dfrac{p}{m}\right)^{|x|}+(1-p)\sum_{i=0}^{|x|-1}
                        \left(\dfrac{p}{m}\right)^{i}\CChi_{I_{x^{-i}}}(t)
                        -m\CChi_{I_{(x,j)}}(t)
                    \Big\} \dd t=0.
            \end{align*}
        \end{re}
        
        To end this section, we proceed with the proof of Theorem \ref{teo:carDN2}.

         \begin{proof}[Proof of Theorem \ref{teo:carDN2}]
            Let $\pi=(x_1,\dots,x_k,\dots)\in\partial\T.$
            Without loss of generality, we can assume that for any $k\in\mathbb{N}$  there is 
            $j_k\in\{0,\dots,m-1\}$ such that $x_{k+1}=(x_k,j_k).$ Then 
            $I_{x_k^{-(|x_k|-1)}}=I_{x_{1}}$ for any $k\in\mathbb{N}.$ Then, by Lemma \ref{lemma:tcl} and Remark \ref{remark:ik0}, we have that
            \[
                p^{-|x_k|}\left(u(x_{k+1})-u(x_k)\right)=
                (1-p)\int_0^1 p^{-|x_k|}\mathcal{K}^{j_k}_m(x_k,t)(g(\psi(\pi))-g(t)) \dd t   
            \]
            where
            \[
                p^{-|x_k|}\mathcal{K}^{j_k}_m(x_k,t)=
                \begin{cases}
                    p\dfrac{m-1}{m-p}-m\dfrac{m-1}{m-p}\left(\dfrac{m}{p}\right)^{|x_k|} &\text{if } 
                    t\in I_{x_{k+1}},\\[7pt]
                     p\dfrac{m-1}{m-p}+m\dfrac{m-1}{m-p}\left(\dfrac{m}{p}\right)^{n(x_k,t)} &\text{if } 
                    t\in I_{x_1}\setminus I_{x_{k+1}},\\[7pt]
                    1   &\text{if } 
                    t\in [0,1]\setminus I_{x_{1}}.
                \end{cases}
            \]
            Then
            \begin{equation}
                \label{eq:DN2.1}
                \begin{aligned}
                    \dfrac{ p^{-|x_k|}}{1-p}&\left(u(\psi(\pi))-u(x_k)\right)=\int_{0}^1 
                    \left[1+\left(p\dfrac{m-1}{m-p}-1\right)\CChi_{I_{x_1}}(t)\right]\left(g(\psi(\pi))-g(t)\right)\dd t\\
                    & \qquad +m\dfrac{1-p}{m-p}\int_{I_{x_1}\setminus I_{x_{k+1}}} 
                    \left(\dfrac{m}{p}\right)^{n(x_k,t)}
                    \left(g(\psi(\pi))-g(t)\right)\dd t\\
                    & \qquad -m\dfrac{1-p}{m-p}\left(\dfrac{m}{p}\right)^{|x_k|}\int_{ I_{x_{k+1}}} 
                    \left(g(\psi(\pi))-g(t)\right)\dd t\\
                    &=\int_{0}^1 
                    \left[1-m\dfrac{1-p}{m-p}\CChi_{I_{x_1}}(t)\right]\left(g(\psi(\pi))-g(t)\right)\dd t+\mathfrak{J}_1(k)+\mathfrak{J}_2(k).
                \end{aligned}
            \end{equation}
            Here,
            $$
             \mathfrak{J}_1(k)=m\dfrac{1-p}{m-p}\int_{I_{x_1}\setminus I_{x_{k+1}}} 
                    \left(\dfrac{m}{p}\right)^{n(x_k,t)}
                    \left(g(\psi(\pi))-g(t)\right)\dd t
            $$
            and
            $$
            \mathfrak{J}_2(k)= -m\dfrac{1-p}{m-p}\left(\dfrac{m}{p}\right)^{|x_k|}\int_{ I_{x_{k+1}}} 
                 \left(g(\psi(\pi))-g(t)\right)\dd t.
            $$

            Since $g\in C^2([0,1]),$ we have that for any $$t\in (\psi(x_{k+1}),\psi(x_{k+1})+\tfrac1{m^{|x_{k}|+1}})$$ 
            there is a $\xi_{k+1}\in(0,1)$ such that
            \[
                g(\psi(\pi))-g(t)=-g'(\psi(\pi))(t-\psi(\pi))-g''(\xi_{k+1})\dfrac{|t-\psi(\pi)|^2}{2}.
            \]
            Then,
            \begin{align*}
             \mathfrak{J}_2(k)=&-m\dfrac{1-p}{m-p}\left(\dfrac{m}{p}\right)^{|x_k|}\int_{ I_{x_{k+1}}} 
                 \left(g(\psi(\pi))-g(t)\right)\dd t\\
                 =&m\dfrac{1-p}{m-p}\left(\dfrac{m}{p}\right)^{|x_k|}\int_{ I_{x_{k+1}}} 
                 \left(g'(\psi(\pi))(t-\psi(\pi))+g''(\xi_{k+1})\dfrac{|t-\psi(\pi)|^2}{2}\right)\dd t\\
                  =&m\dfrac{1-p}{m-p}\left(\dfrac{m}{p}\right)^{|x_k|}
                  \left(
                  g'(\psi(\pi))\int_{ I_{x_{k+1}}}(t-\psi(\pi))\dd t+\int_{ I_{x_{k+1}}} 
                g''(\xi_{k+1})\dfrac{|t-\psi(\pi)|^2}{2}\dd t\right).
            \end{align*}
            Observe that
            \[
                \left|\int_{ I_{x_{k+1}}}(t-\psi(\pi))\dd t\right|= 
                \left|\dfrac{\psi(x_{k+1})-\psi(\pi)}{m^{|x_{k}|+1}}+
                \dfrac{1}{2m^{2|x_{k}|+2}}\right|\le \dfrac{C}{m^{2|x_k|+1}},
            \]
            and
            \[
                \left|\int_{ I_{x_{k+1}}} 
                 g''(\xi_{k})\dfrac{|t-\psi(\pi)|^2}{2}\dd t\right|
                 \le \dfrac{\|g''\|_{\infty}}2  \dfrac1{m^{3|x_k|+3}}.
            \]
            Then, since $\beta>\tfrac1{m+1}$ implies $p>\tfrac1m,$ we have that 
            \begin{equation}
                \label{eq:DN2.2} \lim_{k\to\infty}\mathfrak{J}_2(k)=0.
            \end{equation}
            
            Finally, we compute $$\lim_{k\to \infty}\mathfrak{J}_1(k).$$ 
            Note that 
            \[
                \lim_{k\to \infty} {n(x_k,t)} = \mathcal{N}(\psi(\pi),t) \quad \text{a.e. in } I_{x_1}.
            \]
            On the other hand, by definition of $n(\cdot,\cdot),$ we have that for any $k\in\mathbb{N}$ and $t\in I_{x_1}\setminus\{\psi(\pi)\}$
            \[
                |t-\psi(\pi)|\le \dfrac1{m^{n(x_k,t)}}. 
            \]
            Then, for any $k\in\mathbb{N}$ and $t\in I_{x_1}\setminus\{\psi(\pi)\}$
            \[
                -\log_m(|t-\psi(\pi)|)\ge n(x_k,t),
            \]
            and therefore for any $k\in\mathbb{N}$ and $t\in I_{x_1}\setminus\{\psi(\pi)\}$
            \[
                 \left(\dfrac{m}{p}\right)^{n(x_k,t)}
                    \left|g(\psi(\pi))-g(t)\right| \le
                    \dfrac{\left|g(\psi(\pi))-g(t)\right|}{|t-\psi(\pi)|^{1-\log_m(p)}}\le 
                    \dfrac{\|g'\|_{\infty}}{|t-\psi(\pi)|^{-\log_m(p)}}.
            \]
            
            Since $\tfrac1{m+1}<\beta<\tfrac{1}{2},$ we get $\tfrac{1}{m}<p<1$ and therefore
            $0<-\log_m(p)<1.$ It follows that
            \[
                \dfrac{1}{|t-\psi(\pi)|^{-\log_m(p)}}\in L^1(I_{x_1}).
            \]
            So, by dominated dominated convergence theorem
            \begin{equation}
                \label{eq:DN2.3}
                 \lim_{k\to \infty}\mathfrak{J}_1(k)=m\dfrac{1-p}{m-p}
                 \int_{I_{x_1} }
                    \left(\dfrac{m}{p}\right)^{\mathcal{N}(\psi(\pi),t)}
                    \left(g(\psi(\pi))-g(t)\right)\dd t.
            \end{equation}
             
             Thus, by \eqref{eq:DN2.1},  \eqref{eq:DN2.2}, and \eqref{eq:DN2.3}, we get
             \begin{align*}
                   \lim_{k\to\infty} & p^{-|x_k|}\left(u(\psi(\pi))-u(x_k)\right)=\\
                    &=(1-p)\int_{0}^1 
                    \left[1+m\dfrac{1-p}{m-p}
                    \left(\left(\dfrac{m}{p}\right)^{\mathcal{N}(\psi(\pi),t)}-1\right)
                    \CChi_{I_{x_1}}(t)\right]\left(g(\psi(\pi))-g(t)\right)\dd t.
                    \qedhere
            \end{align*}
        \end{proof}
         
        \begin{re} 
        	For any $\pi=(x_1,\dots,x_k, \dots)\in \partial\T$
        	one may be tempted to compute
        	\begin{align*}
                  \lim_{k\to\infty} 
                  	m^{|x_k|} &\left(u(\psi(\pi))-u(x_k)\right)=\\
                  	&=\lim_{k\to\infty} 
                  	(mp)^{|x_k|} (1-p)\int_{0}^1 
                    \left[1+\left(p\dfrac{m-1}{m-p}-1\right)
                    \CChi_{I_{x_1}}(t)\right]\left(g(\psi(\pi))-g(t)\right)\dd t\\
                    & \qquad 
                    +m\dfrac{(1-p)^2}{m-p}
                    (mp)^{|x_k|}\int_{I_{x_1}\setminus I_{x_{k+1}}} 
                    \left(\dfrac{m}{p}\right)^{n(x_k,t)}
                    \left(g(\psi(\pi))-g(t)\right)\dd t\\
                    & \qquad -m\dfrac{(1-p)^2}{m-p}m^{2|x_k|}\int_{ I_{x_{k+1}}} 
                    \left(g(\psi(\pi))-g(t)\right)\dd t.
                \end{align*}
            Unfortunately, we are unable to compute this limit in this way. This is due to the fact that, 
        	in general, the following limit 
        	\begin{align*}
				&\lim_{k\to\infty} 
					m^{2|x_k|}\int_{ I_{x_{k+1}}} 
                    \left(g(\psi(\pi))-g(t)\right)\dd t=\\
                    &=\lim_{k\to\infty} 
                	m^{|x_k|}g'(\psi(\pi))
                		 \dfrac{\psi(x_{k+1})-\psi(\pi)}{m}
                    	 +\dfrac{g'(\psi(\pi))}{2m^{2}}
                    	 +m^{2|x_k|} \int_{ I_{x_{k+1}}} 
                	g''(\xi_{k+1})\dfrac{|t-\psi(\pi)|^2}{2}\dd t
           \end{align*}
           does not exists.
        \end{re}
        
        \medskip

{\bf Acknowledgements.} Supported by CONICET grant PIP GI No 11220150100036CO
(Argentina), by  UBACyT grant 20020160100155BA (Argentina) and by MINECO MTM2015-70227-P
(Spain).


\end{document}